\documentclass[graybox]{svmult}

\usepackage{mathtools}       
\usepackage{mathptmx}       
\usepackage{helvet}         
\usepackage{courier}        
\usepackage{type1cm}        
\usepackage{makeidx}         
\usepackage{graphicx}        
\usepackage{multicol}        
\usepackage[bottom]{footmisc}

\usepackage{url}
\usepackage{amstext, amsmath, amssymb}
\usepackage{booktabs}
\usepackage{fancyvrb}
\usepackage{stmaryrd}
\usepackage{pdfsync}
\usepackage{algorithm}
\usepackage{tikz}
\usepackage{pgfplots}
\usepackage{pgfplotstable}
\usepackage{subfig}
\usepackage{tabularx}
\usepackage{todonotes}
\usepackage{verbatim}
\usepackage[numbers,square,sort&compress]{natbib}

\usepackage[plainpages=false]{hyperref}
\hypersetup{
		bookmarksopen=true,
		bookmarksnumbered=true,
		pdfborder={0 0 0},
		colorlinks=true,
		linkcolor=blue,
		urlcolor=blue	
}

\newcommand{\RR}{\mathbb{R}}

\newcommand{\mcP}{\mathcal{P}}

\newcommand{\mcV}{\mathcal{V}}
\newcommand{\mcK}{\mathcal{K}}
\newcommand{\mcE}{\mathcal{E}}

\newcommand{\mcF}{\mathcal{F}}

\newcommand{\mcA}{\mathcal{A}}

\newcommand{\mcT}{\mathcal{T}}

\newcommand{\mcO}{\mathcal{O}}

\newcommand{\tn}{|\mspace{-1mu}|\mspace{-1mu}|}

\newcommand{\jump}[1]{[#1]}
\newcommand{\mean}[1]{\langle {#1} \rangle}

\newcommand{\Gammah}{{\Gamma^h}}
\newcommand{\nablas}{\nabla_\Gamma}
\newcommand{\nablash}{\nabla_{\Gamma^h}}

\newcommand{\foralls}{\forall\,}

\newcommand{\dsh}{\,\mathrm{d} \Gamma^h}

\newcommand{\ub}{u_{\Omega}}
\newcommand{\us}{u_{\Gamma}}
\newcommand{\ns}{n_{\Gamma}}
\newcommand{\vb}{v_{\Omega}}
\newcommand{\vs}{v_{\Gamma}}
\newcommand{\wb}{w_{\Omega}}
\newcommand{\ws}{w_{\Gamma}}
\newcommand{\ks}{k_{\Gamma}}

\newcommand{\cs}{c_{\Gamma}}
\newcommand{\cb}{c_{\Omega}}
\newcommand{\fs}{f_{\Gamma}}
\newcommand{\fb}{f_{\Omega}}

\newcommand{\Vs}{V_{\Gamma}}
\newcommand{\Vb}{V_{\Omega}}
\newcommand{\as}{a_{\Gamma}}
\newcommand{\ab}{a_{\Omega}}
\newcommand{\abs}{a_{\Omega\Gamma}}
\newcommand{\As}{A_{\Gamma}}
\newcommand{\Ab}{A_{\Omega}}
\newcommand{\ls}{l_{\Gamma}}
\newcommand{\lb}{l_{\Omega}}
\newcommand{\jsh}{j_{\Gamma}^h}
\newcommand{\jbh}{j_{\Omega}^h}

\newcommand{\mcFb}{\mcF_{\Omega}}
\newcommand{\mcFbg}{\mcF_{\Omega, g}}
\newcommand{\mcFs}{\mcF_{\Gamma}}
\newcommand{\mcTb}{\mcT_{\Omega}}
\newcommand{\mcTs}{\mcT_{\Gamma}}
\newcommand{\tnb}{\tn_{h, \Omega}}
\newcommand{\tns}{\tn_{h, \Gamma}}
\newcommand{\gammab}{\gamma_{\Omega}}
\newcommand{\gammas}{\gamma_{\Gamma}}
\newcommand{\mub}{\mu_{\Omega}}
\newcommand{\mus}{\mu_{\Gamma}}
\newcommand{\taub}{\tau_{\Omega}}
\newcommand{\taus}{\tau_{\Gamma}}

\newcommand{\onehalf}{\frac{1}{2}}

\DeclareMathOperator{\spann}{span}
\DeclareMathOperator{\dist}{dist}
\DeclareMathOperator{\diam}{diam}

\DefineVerbatimEnvironment{code}{Verbatim}{frame=single,rulecolor=\color{blue}}

\newcommand\amnote[2][]{\todo[inline, caption={2do}, color=cyan!40 #1]{
\begin{minipage}{\textwidth-4pt}\underline{AM:} #2\end{minipage}}}

\makeindex             

\begin{document}

\title*{A Cut Discontinuous Galerkin Method for Coupled Bulk-Surface Problems}

\author{Andr\'e Massing}
\institute{Andr\'e Massing \at Department of Mathematics and Mathematical Statistics,
 Ume\aa{} University, SE-90187 Ume\aa{}, Sweden, \email{andre.massing@umu.se}}
%
%
\maketitle


\abstract{
  We develop a cut Discontinuous Galerkin method (cutDGM) for a
  diffusion-reaction equation in a bulk domain which is coupled to a
  corresponding equation on the boundary of the bulk domain.
  The bulk domain is embedded into a structured, unfitted background mesh.
  By adding certain stabilization terms to the discrete variational
  formulation of the coupled bulk-surface problem,
  the resulting cutDGM is provably stable and exhibits
  optimal convergence properties as demonstrated by numerical experiments.
  We also show both theoretically and numerically
  that the system matrix is well-conditioned, irrespective of
  the relative position of the bulk domain in the background mesh.
}

\section{Introduction}
\label{sec:introduction}

In recent years, the analysis and numerical solution of coupled bulk-surface
partial differential equations (PDE) have gained a large interests
in the fields of computational engineering and scientific computing.
Indeed, a number of important phenomena in biology, geology and
physics can be described by such PDE systems.
A prominent use case are flow and transport problems in porous media when
large-scale fracture networks are modeled as 2D geometries embedded into a 3D bulk
domain~\cite{MartinJaffreRoberts2005,FormaggiaFumagalliScottiEtAl2013}.
Another important example is the modeling of cell motility where
reaction-diffusion systems on the cell membrane and inner cell are coupled to describe the
active reorganization of the cytoskeleton~\cite{NovakGaoChoiEtAl2007,Raetz2015}.
Coupled bulk-surface PDEs arise also naturally when modeling
incompressible multi-phase flow problems with surfactants~\cite{GanesanTobiska2009,
  GrossReusken2011,MuradogluTryggvason2008,GrosReusken2013}.

The numerical solution of coupled bulk-surface systems
poses several challenges even for modern computational methods.
First, one faces a system of
coupled PDEs on domains of different topological dimensionality,
which needs to be accommodated by the numerical method at hand.
Second, extremely complex surface geometries naturally appear in many realistic
application scenarios, e.g., when complex fracture networks in porous media models 
are considered, and thus fast and robust mesh generation becomes a challenge.
Moreover, the simulation of complex droplet systems shows that, even if the
initial surface geometry is relatively simple, it might evolve significantly
over time and thus can undergo large or even topological changes.
For traditional discretization methods, a costly remeshing of the computational
domain is then the only resort, and the question of how
to transfer the computed solution components between different meshes
efficiently and accurately becomes an urgent and challenging matter.

As a potential remedy to these challenges, the
so-called cut finite element method (CutFEM) has gained a large interest
in recent years, see~\cite{BurmanClausHansboEtAl2014} for a review.
The basic idea is to decouple the description
of the geometry as much as possible from the underlying approximation spaces
by embedding the geometry of the domain into 
a fixed background mesh which is also used to construct the finite
element spaces for the surface and bulk approximations.
In order to obtain a stable
method, independent of the position of the geometry in the background
mesh, and to handle the potential small cut elements in the analysis, certain
stabilization terms are added that provide control of the local
variation of the discrete functions.
In this work we extend ideas from CutFEM framework developed
over the last half a decade to synthesize a novel cut discontinuous Galerkin method
(cutDGM) for coupled bulk-surface PDEs.

\subsection{Earlier work}
The development of the cut finite element framework was initiated by the seminal
papers~\cite{BurmanHansbo2010,BurmanHansbo2012} considering the weak imposition
of boundary conditions for the Poisson problem on unfitted meshes.
Shortly after, the idea was picked up by a number of authors
to formulate cut finite element methods
for the Stokes type problems\cite{BurmanHansbo2013,Massing2012,MassingLarsonLoggEtAl2013a,BurmanClausMassing2015,GuzmanOlshanskii2016,HansboLarsonZahedi2014a},
the Oseen problem~\cite{MassingSchottWall2016,WinterSchottMassingEtAl2017}
and number of related
fluid problems, see~\cite{Schott2017} for a comprehensive overview.

Prior to the arrival of CutFEMs,
unfitted \emph{discontinuous} Galerkin methods have successfully
been employed to solve boundary and interface problems
on complex and evolving domains~\cite{BastianEngwer2009, Saye2015},
including two-phase flows~\cite{SollieBokhoveVegt2011,HeimannEngwerIppischEtAl2013,MuellerKraemer-EisKummerEtAl2016}.
In unfitted discontinuous Galerkin method,
troublesome small cut elements can be merged with neighbor elements with a large intersection support
by simply extending the local finite element basis from the large element to the small cut element.
As the inter-element continuity is enforced only weakly,
the coupling of the these extended basis functions to additional elements incident
with the small cut elements does not lead to an
over-constrained system, as it would happen if globally continuous finite element functions were
employed.
Consequently, unfitted discontinuous Galerkin methods
provide an alternative stabilization mechanism to ensure the
well-posedness and well-conditioning of the discretized systems.
Thanks to their favorable conservation and stability properties,
unfitted discontinuous Galerkin methods
remain an attractive alternative to continuous CutFEMs,
but some drawbacks are the almost complete absence of numerical
analysis except for~\cite{Massjung2012,JohanssonLarson2013}, the implementational labor to
reorganize the matrix sparsity patterns when agglomerating cut elements,
and the lack of natural discretization approaches for PDEs defined on surfaces.

For PDEs defined on surfaces, the idea of using the finite element space from the embedding
bulk mesh was already formulated and analyzed in~\cite{OlshanskiiReuskenGrande2009},
and then further extended to high-order methods~\cite{GrandeReusken2016} and
evolving surface problems~\cite{OlshanskiiReuskenXu2014a,HansboLarsonZahedi2015b}.
A stabilized cut finite element for the Laplace-Beltrami problem
were introduced in~\cite{BurmanHansboLarson2015} where the additional stabilization
cures the resulting system matrix from being ill-conditioned,
as an alternative to diagonal preconditioning used in~\cite{OlshanskiiReusken2010}.
Finally, after the initial work~\cite{ElliottRanner2013} on fitted finite element
discretizations of coupled bulk-surface PDEs, only a
few number of corresponding unfitted (continuous) finite element schemes have been formulated,
see~\cite{BurmanHansboLarsonEtAl2014,HansboLarsonZahedi2016,GrossOlshanskiiReusken2014}.

\subsection{Contribution and outline of the paper}
In this work, we formulate a novel cut discontinuous Galerkin method
for the discretization of coupled bulk-surface problems on a given bounded domain $\Omega$.
The strong and weak formulation of a continuous prototype problem are briefly reviewed in Section~2.
Motivated by our
earlier work~\cite{BurmanHansboLarsonEtAl2016a},
we introduce a cut discontinuous Galerkin method for bulk-surface PDEs in Section~3.
The method is employs
discontinuous piecewise linear elements on a background mesh consisting of simplices in $\mathbb{R}^d$.
The boundary $\Gamma$ of the computational domain $\Omega$ is represented by a continuous, piecewise
approximation of distance functions associated with $\Gamma$.
For both the discrete bulk and surface domain, the active background meshes consist
of those elements with a non-trivial intersection with the respective domain.
Utilizing the general stabilization framework developed for continuous CutFEMs,
we add certain, so-called ghost penalty stabilization in the vicinity of the embedded
surface to ensure that the overall cutDGM is stable and its system matrix is well-conditioned.
The exact mechanism is further elucidated in Section~4, where short proofs of
the coercivity of the bilinear forms introduced in Section~3 are given.
We also demonstrate that the condition number of the (properly
rescaled) system matrix scales like~$\mcO(h^{-2})$.  All theoretical
results hold with constants independent of the position of the
domain relative to the background mesh.

While a full a priori analysis of the proposed method is beyond the limited scope of this
work, we perform a convergence rate study in Section~5 instead, demonstrating the optimal approximation
properties of the formulated cutDGM. Finally, we also demonstrate that the employed CutFEM
stabilizations are essential for the geometrically robust convergence and conditioning
properties of the method.

\subsection{Basic notation}
Throughout this work, $\Omega \subset \RR^d$ denotes an
open and bounded domain with smooth boundary 
$\Gamma = \partial \Omega$. For $U \in \{\Omega, \Gamma \}$ and $ s \in \RR$,
let $H^{s}(U)$ be the
standard Sobolev spaces defined on $U$.
As usual, we write $(\cdot,\cdot)_{s,U}$ and $\|\cdot\|_{s,U}$ for the associated inner products
and norms. If there is no confusion, we
occasionally write $(\cdot,\cdot)_{U}$ and $\|\cdot \|_{U}$ for the
inner products and norms associated with $L^2(U)$, with $U$ being a
measurable subset of $\RR^d$. 
Finally, any norm $\|\cdot\|_{\mcP_h}$ used in this work which
involves a collection of geometric entities $\mcP_h$ should be
understood as broken norm defined by $\|\cdot\|_{\mcP_h}^2 =
\sum_{P\in\mcP_h} \|\cdot\|_P^2$ whenever $\|\cdot\|_P$ is well-defined,
with a similar convention for scalar products $(\cdot,\cdot)_{\mcP_h}$.
Finally, it is understood that the notation $\|\cdot\|_{\mcP_h \cap U}$,
for any given set $U\subset \RR^d$
means to sum up over the corresponding cut parts; that is,
$\|\cdot\|_{\mcP_h \cap U}^2 =
\sum_{P\in\mcP_h} \|\cdot\|_{P\cap U}^2$.

\section{Model problem}
Let $\Omega \subset \RR^d$ be a bounded domain with smooth boundary $\Gamma$
equipped with a outward pointing normal field $n_{\Gamma}$
and signed distance function $\rho$; that is, $\rho$  satisfies $ \rho(x) = \pm \dist(x, \Gamma)$
with the distance being strictly negative if $x \in \Omega$ and positive
otherwise. It is well known that for some positive $\delta_0$ small enough
and any $\delta$ with $0 < \delta < \delta_0$,
every point~$x$ in the tubular neighborhood
$U_{\delta}(\Gamma) = \{x \in \RR^d: |\rho(x)| < \delta\}$
has a uniquely defined closest point $p(x)$ on $\Gamma$
satisfying $x = p(x) + \rho(x)n(p(x))$, see, e.g,~\cite[Sec. 14.6]{GilbargTrudinger2001}.
For any function $v_{\Gamma} \in C^1(\Gamma)$, the tangential gradient
$\nabla_\Gamma v_{\Gamma}$ is defined by
\begin{align}
 \nabla_\Gamma v_{\Gamma}  = P \nabla v_{\Gamma},
\end{align}
with $P(x) = I - \ns(x)\otimes\ns(x)$ denoting the projection of $\RR^d$ onto
the tangential space at point $x \in \Gamma$.
As model for a coupled bulk-surface problem,
we consider the problem: given functions $f_{\Omega}$ and $f_{\Gamma}$
on $\Omega$ and $\Gamma$, respectively, and positive constants
$c_{\Omega}, c_{\Gamma}$,
find functions $u_{\Omega} : \Omega \to \RR$
and $u_{\Gamma}: \Gamma \to \RR$ such that
\begin{subequations}
  \label{eq:bulk-surface-strong}
\begin{alignat}{3}
  - \Delta \ub + \ub &= \fb & &\quad \text{in } \Omega,
  \label{eq:bulk-strong}
  \\
 \partial_n \ub &=  \cs \us - \cb \ub \ && \quad \text{on } \Gamma,
  \label{eq:surface-bulk-coupling-strong}
  \\
  - \Delta_{\Gamma}\us + \us &= \fs - \partial_n \ub & &\quad \text{on }
  \Gamma,
  \label{eq:surface-strong}
\end{alignat}
\end{subequations}
where $\Delta_\Gamma$ is the Laplace-Beltrami operator on $\Gamma$
defined by
\begin{equation}
\Delta_\Gamma = \nabla_\Gamma \cdot \nabla_\Gamma.
\end{equation}
Following~\cite{ElliottRanner2013,BurmanHansboLarsonEtAl2014},
we can derive a weak formulation by multiplying~(\ref{eq:bulk-strong}) with a test function
$\vb \in H^1(\Omega)$ and using Green's formula to obtain
\begin{align}
  (\nabla \ub, \nabla \vb)_{\Omega} - (\partial_n \ub, \vb )_{\Gamma}
  + (\ub, \vb)_{\Omega}
  = (f, \vb)_{\Omega},
\end{align}
which together with the coupling condition~(\ref{eq:surface-bulk-coupling-strong})
leads to
\begin{align}
  (\nabla \ub, \nabla \vb)_{\Omega} 
  + (\ub, \vb)_{\Omega}
  + (\cb \ub - \cs \us, \vb )_{\Gamma}
  = (\fb, \vb)_{\Omega}.
  \label{eq:bulk-weak}
\end{align}
Next, taking $\vs \in H^1(\Gamma)$,
a similar treatment of~(\ref{eq:surface-strong}) yields
\begin{align}
  (\nabla \us, \nabla \vs)_{\Gamma}
+ (\us, \vs)_{\Gamma}
  - (\cb \ub - \cs \us, \vs )_{\Gamma} =
  (\fs, \vs)_{\Gamma}.
  \label{eq:surface-weak}
\end{align}
Now replacing $\vb$ with $\cb \vb$ in~(\ref{eq:bulk-weak}) and $\vs$ with $\cs \vs$
in~(\ref{eq:surface-weak})
and summing up the two equations motivates us to
introduce the following forms to describe the bulk, surface and coupling related parts
of the overall bilinear form $a(\cdot, \cdot)$:
\begin{align}
  a_{\Omega}(\ub, \vb)
  &= (\nabla \ub, \nabla \vb)_{\Omega} + (\ub, \vb)_{\Omega},
    \label{eq:a-bulk-def}
  \\
  a_{\Gamma}(\us, \vs)
  &= (\nablas \us, \nablas \vs)_{\Gamma} + (\us, \vs)_{\Gamma},
    \label{eq:a-surface-def}
  \\
  a_{\Omega\Gamma}(u, v)
  &= (\cb \ub - \cs \us, \cb \vb - \cs \vs)_{\Gamma}.
  \label{eq:a-bulk-surface-def}
\end{align}
As final ingredient, we
define the bulk function spaces $\Vb = H^1(\Omega)$, the surface function space $\Vs =  H^1(\Gamma)$ and
the total space $V = \Vb \times \Vs$,
and introduce also the short-hand notation $u = (\ub, \us) \in V$ and $v = (\vb, \vs) \in V$.
Then the variational problem for the coupled bulk-surface PDE~(\ref{eq:bulk-surface-strong}) is to seek
$u \in V$ such that $\foralls v \in V$
\begin{align}
  a(u, v) = l(v),
 \label{eq:bsp-weak} 
\end{align}
where the bilinear form $a(\cdot, \cdot)$ and linear form $l(\cdot)$ are given by
\begin{align}
  a(u, v) &= 
             \cb \ab (\ub, \vb)
            + \cs \as (\us, \vs)
            + a_{\Omega\Gamma}(u, v),
  \label{eq:a-def}
  \\
  l(v) &= \cb (fb, \vb)_{\Omega} + \cs (\fs, \vs)_{\Gamma}.
  \label{eq:l-def}
\end{align}
Using the natural energy norm $\tn v \tn = \sqrt{a(v,v)}$, it follows immediately that
the bilinear form $a$ is coercive with respect to $\tn \cdot \tn$ and that both forms $a$ and $l$
are continuous, and thus the Lax-Milgram theorem ensures the existence of a unique solution to
the weak problem~\eqref{eq:bsp-weak}, see also \cite{ElliottRanner2013}.

\section{A cut discontinuous Galerkin method for bulk-surface problems}
\label{sec:cutDGM}
The main idea in the cut discontinuous Galerkin discretization of
the bulk-surface PDE~(\ref{eq:bsp-weak})
is now to embedd the domain $\Omega$ into an easy-to-generate 3d background mesh
in an unfitted manner.
The approximation spaces for the discrete bulk and surface solution components 
are then given by suitable restrictions of the discontinuous finite element functions
defined on background mesh to the bulk and surface domains, respectively.
We start with describing the relevant computational domains and related geometric
quantities before we turn to the definition of the cut finite element spaces
and the final discrete formulation.

\subsection{Computational domains}
\label{ssec:comp-domain}
Assume that ${\mcT}^h$
is a quasi-uniform\footnote{Quasi-uniformity is mainly assumed
  to simplify the overall presentation.} background mesh
with global mesh size $h$ consisting of shape-regular elements $\{T\}$
which cover $\Omega$.
Let $\rho^h$ be a continuous, piecewise
linear approximation of the distance function $\rho$
and define the discrete surface $\Gamma^h$ as the zero level set of
$\rho^h$,
\begin{align}
  \Gamma^h &= \{ x \in \Omega : \rho^h(x) = 0 \}
  \intertext{and correspondingly, the discrete bulk domain is given by}
  \Omega^h &= \{ x \in \Omega : \rho^h(x) < 0 \}.
\end{align}
Note that $\Gamma^h$ is a polygon consisting of flat faces
with a piecewise defined constant exterior unit normal $n^h$.
We assume that: 
\begin{itemize}
\item $\Gamma^h \subset U_{\delta_0}(\Gamma)$ and that the closest
point mapping $p:\Gamma^h \rightarrow \Gamma$ is a bijection for $0< h
\leq h_0$.
\item The following estimates hold
  \begin{equation}
    \| \rho \|_{L^\infty(\Gamma^h)} \lesssim h^2, \qquad
    \| n - n^h \circ p \|_{L^\infty(\Gamma)} \lesssim h.
\label{eq:geometric-assumptions-II}
\end{equation}
\end{itemize}
These properties are, for instance, satisfied if
$\rho^h$ is the Lagrange interpolant of $\rho$.
Starting from the  background mesh $\mcT^h$, we define the
\emph{active} (background) meshes for discretization of the bulk and surface problem by
\begin{align} 
  \mcTb^h &= \{ T \in {\mcT}^{h} : T^{\circ} \cap \Omega^h \neq \emptyset \},
  \label{eq:active-mesh-bulk}
  \\
  \mcTs^h &= \{ T \in \mcTb^{h} : T \cap \Gamma^h \neq \emptyset \},
  \label{eq:active-mesh-surface}
\end{align}
respectively. 
Here, $T^{\circ}$ denotes the topological interior of an element $T$
and thus $\mcTb^h$ does not contain any element which intersects only with the
boundary $\Gamma^h$ but not with the interior $\Omega^h$.
Clearly, $\mcTs^h \subset \mcTb^h$. 
For the actives meshes $\mcTb^h$ and $\mcTs^h$, the corresponding sets of interior faces are denoted by
\begin{align} 
  \mcFb^h &= \{ F =  T^+ \cap T^-: T^+, T^- \in \mcTb^h \},
  \label{eq:faces-interior-bulk}
  \\
  \quad
  \mcFs^h &= \{ F =  T^+ \cap T^-: T^+, T^- \in \mcTs^h \}.
  \label{eq:faces-interior-surface}
\end{align}
Note that by extracting
$\mcTs^h$ from $\mcTb^h$ instead of $\mcT^h$, we automatically pick a unique
element from $\mcT^h$ in the case that $\Gamma^h \cap T$ coincides with
an interior face of the background mesh~$\mcT^h$.
Additionally, we will also need the set of interior faces of the active bulk mesh $\mcTb^h$
which belong to elements intersected by the discrete surface~$\Gamma^h$,
\begin{align} 
  \mcFb^{h,g} = \{ F = T^+ \cap T^- \in \mcFb^h:  T^+ \in \mcTs^h \lor T^- \in \mcTs^h \}.
  \label{eq:faces-gp}
\end{align}
This set of faces will be instrumental in defining certain stabilization forms, also
known as \emph{ghost penalties}, hence the superscript $g$.
As usual, face normals $n^+_F$ and $n^-_F$ are given by the unit
normal vectors which are perpendicular on $F$ and are pointing exterior
to $T^+$ and $T^-$, respectively.

For the surface approximation $\Gamma^h$, 
corresponding collection of geometric entities can be generated by considering
the intersection of $\Gamma^h$ with individual elements of the active mesh, i.e., we
define the set of surface faces and their edges by
\begin{align}
  \mcK^h&=\{K = \Gamma^h \cap T : T \in \mcT_{\Gamma}^h \},
  \\
  \mcE^h&=\{E = K^+ \cap K^-:  K^+, K^- \in \mcK^h \}.
\end{align}
To each interior edge $E$ we associate the co-normals $n^{\pm}_E$ 
given by the unique unit vector which is
coplanar to the surface element $K^{\pm}$, perpendicular
to $E$ and points outwards with respect to $K^{\pm}$. 
Note that while the two face normals $n_F^{\pm}$ only differ by a sign,
the edge co-normals $n_E^{\pm}$ do lie in genuinely different planes.
The various set of geometric entities are illustrated in
Figure~\ref{fig:domain-set-up}.
\begin{figure}[htb]
  \begin{center}
  \begin{minipage}[t]{0.45\textwidth}
    \vspace{0pt}
    \includegraphics[width=1.0\textwidth]{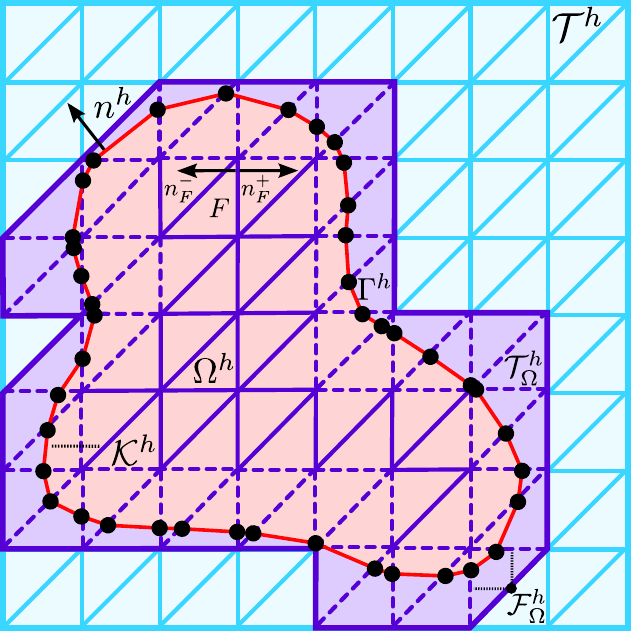}
  \end{minipage}
  \hspace{0.02\textwidth}
  \begin{minipage}[t]{0.45\textwidth}
    \vspace{0pt}
    \includegraphics[width=1.0\textwidth]{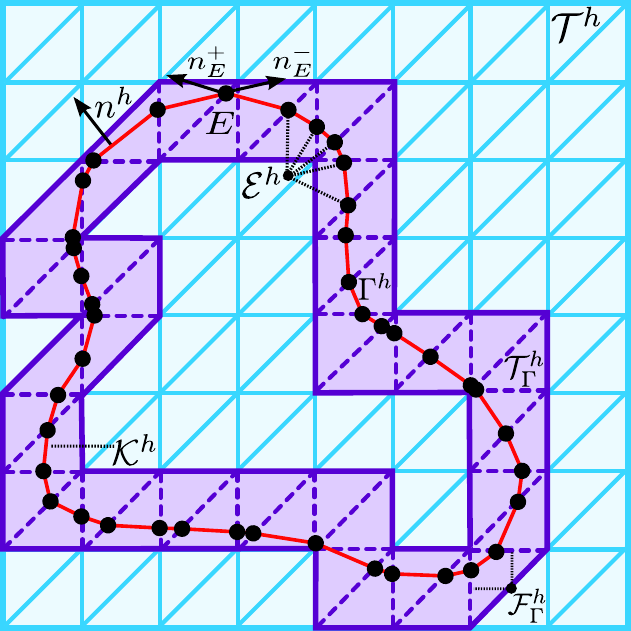}
  \end{minipage}
\end{center}
\caption{Computational domains for the bulk-surface problem. (Left) Active
  mesh used to define the approximation space for
  the bulk solution. Faces on which ghost penalty stabilization are defined are
  plotted as dashed faces. (Right) Corresponding computational domain set-up for the
discretization of the surface.}
  \label{fig:domain-set-up}
\end{figure}
%
%
%

\subsection{The cut discontinuous Galerkin method}
\label{ssec:cutDGM}
We start with defining the discrete counterparts of the function
spaces $\Vb$ and $\Vs$ to be
the broken polynomial spaces consisting of piecewise linear, but not
necessarily globally continuous functions defined on the respective
active meshes:
\begin{align}
  \Vb^h = \bigoplus_{T \in \mcTb^h} P_1(T),
  \quad
  \Vs^h = \bigoplus_{T \in \mcTs^h} P_1(T),
  \quad
  V^h = \Vb^h \times \Vs^h.
\label{eq:Vsh-def}
\end{align}
For the formulation of the cut discontinuous Galerkin method,
we also need the notation of average and fluxes of piecewise defined functions.
More precisely, assume that
$\sigma$ and $w$ are, possibly vector-valued, elementwise defined functions
on $\mcT^h$ which are smooth enough to admit a two-valued trace on all faces.
Then the standard and face normal weighted average fluxes 
are given by
\begin{align}
\mean{\sigma}|_F &= \dfrac{1}{2} (\sigma_F^{+} + \sigma_F^{-}),
  \label{eq:mean-std-def-F}
\\
  \mean{n_F  \cdot \sigma }|_F &=
                                 \dfrac{1}{2} n^{+}_F \cdot (\sigma_F^{+} + \sigma_F^{-})
                                 =
                                 \dfrac{1}{2}(n^{+}_F \cdot \sigma_F^{+} - n^{-}_F \cdot \sigma_F^{-}),
  \label{eq:mean-def-F}
\end{align}
while the jump across an interior face $F \in
\mcF^h$ is defined by
\begin{align}
\jump{w}|_F &= w_F^{+} - w_F^{-},
\end{align}
with
$w(x)^\pm = \lim_{t\rightarrow 0^+} w(x - t n_F^{\pm})$.
In the case of vector-valued functions, the jump is taken componentwise.
As the co-normal vectors $n_E^{\pm}$ are generally not collinear,
the standard and co-normal weighted average fluxes for a piecewise discontinuous,
possibly vector-valued function
$\sigma$ on $\mcK^h$ is defined by
\begin{align}
\mean{\sigma}|_E &= \dfrac{1}{2} (\sigma_E^{+} + \sigma_E^{-}),
  \label{eq:mean-std-def}
\\
\mean{n_E  \cdot \sigma }|_E &= \dfrac{1}{2}(n^{+}_E \cdot \sigma_E^{+} - n^{-}_E \cdot \sigma_E^{-}),
  \label{eq:mean-def}
\end{align}
respectively. Similarly, the jump across an interior face $E\in \mcE^h$ is
given by
\begin{align} 
\jump{w}|_E &= w_E^{+} - w_E^{-}.
\end{align}
We are now ready to define the discrete, discontinuous Galerkin
counterparts of the bilinear forms \eqref{eq:a-bulk-def},
\eqref{eq:a-surface-def}, and~(\ref{eq:a-bulk-surface-def}) and set
\begin{align}
  \ab^h(\vb,\wb) &= (\nabla \vb, \nabla \wb)_{\mcTb^h\cap\Omega^h}
                   + (\vb, \wb)_{\mcTb^h\cap\Omega^h}
  + \gammab (h^{-1}\jump{\vb},\jump{\wb})_{\mcFb^h}
  \nonumber
 \\                    
  &\qquad  
  - (\mean{n_F \cdot \nabla \vb }, \jump{\wb})_{\mcFb^h\cap\Omega^h}
- (\jump{\vb}, \mean{n_F \cdot \nabla \wb })_{\mcFb^h\cap\Omega^h},
  \label{eq:ah-bulk-def}
  \\
  \as^h(\vs,\ws) &= (\nablash \vs, \nablash \ws)_{\mcK^h}
                   + (\vs, \ws)_{\mcTs^h\cap\Gamma^h}
                 + \gammas (h^{-1}\jump{\vs},\jump{\ws})_{\mcE^h}
  \nonumber
  \\
               &\qquad
  - (\mean{n_E \cdot \nabla \vs }, \jump{\ws})_{\mcE^h}
  - (\jump{\vs}, \mean{n_E \cdot \nabla \ws })_{\mcE^h},
    \label{eq:ah-surface-def}
    \\
  a_{\Omega\Gamma}^h(v, w)
                 &= (\cb \vb - \cs \vs, \cb \wb - \cs \ws)_{\Gamma^h},
  \label{eq:ah-bulk-surface-def}
  \\
  a^h(v, w)  &= \cb\ab^h(\vb,\wb) + \cs\as^h(\vs, \ws)
  + a^h_{\Omega\Gamma}(v, w).
  \label{eq:ah-def}
\end{align}
Similarly, the relevant discrete linear forms are given by
\begin{align}
  \lb^h(\vb) &= (\fb, \vb)_{\Omega^h},
  \label{eq:lbh-def}
  \\
  \ls^h(\vs) &= ( \fs^e, \vs )_{\Gamma^h},
  \label{eq:lsh-def}
  \\
  l^h(v) &= \cb \lb^h(\vb) +  \cs \ls^h (\vs ).
  \label{eq:lh-def}
\end{align}
Here, $\fs^e$ denotes the extension of $\fs$ to the tubular
neighborhood $U_{\delta}(\Gamma)$ using the closest point projection
by requiring that $\fs^e(x) = \fs(p(x))$.
Finally, appropriate ghost-penalties for the bulk and surface part are defined by
\begin{align}
  \jbh (\vb,\wb) &= \mub
                       h^{-1}(\jump{\vb},\jump{\wb})_{\mcFbg^h}
                       + \taub h (  n_F \cdot \jump{\nabla \vb} ,n_F \cdot  \jump{\nabla \wb} )_{\mcFbg^h}
  \label{eq:jbh-def}
  \\
  \jsh (\vs,\ws) &= \mus
  h^{-2}(\jump{\vs},\jump{\ws})_{\mcFs^h}
                       + \taus ( n_F \cdot \jump{\nabla \vs} ,n_F \cdot  \jump{\nabla \ws} )_{\mcFs^h}
               \label{eq:jsh-def}
  \\
  j^h(v,w) &= \cb \jbh(\vb, \wb) + \cs\jsh(\vs, \ws)
  \label{eq:jh-def}
\end{align}
where $\mub,\mus, \taub, \taus$ are positive parameters. To ease the
notation, we also define the ghost penalty
enhanced bulk and surface  bilinear forms
\begin{align}
  A_U^h(v_U, w_U) = a_U^h(v_U, w_U) + j_U^h(v_U, w_U),  \quad  U \in \{\Omega, \Gamma\}.
\end{align}
Now the cut discontinuous Galerkin method for the bulk-surface problem is to
seek $u^h = (\ub^h, \us^h) \in V^h = \Vb^h \times \Vs^h$ such that
$\foralls v \in V^h$
\begin{align}
  A^h(u^h, v) \coloneqq a^h(u^h,v) + j^h(u^h, v) = l^h(v).
  \label{eq:sbp-cutdg-formulation}
\end{align}
\begin{remark}
  The defined ghost penalties are crucial to devise a geometrically robust,
  well-conditioned and optimally convergent discretization method, irrespective
  of the particular cut configuration. We note that in general, the unstabilized
  cutDGM suffers from three drawbacks. First, certain inverse inequalities
  fundamental for the analysis of DGMs do not hold any more when only the
  physical, cut part of the background mesh is considered. Second,
  cut configurations with very small cut parts can lead to an almost vanishing
  contribution of certain degree of freedoms in the system matrix.
  Third, the restriction of discontinuous finite element functions from the active
  mesh $\mcTs^h$ to the surface $\Gamma^h$
  results in a highly linear dependent set of functions, and thus purely surface-based
  ``norms''are not capable of distinguishing them, which also leads
  to an ill-conditioned system matrix.
\end{remark}

\section{Stability properties}
\label{sec:stability-properties}
In this section, we investigate the stability properties of the
proposed cutDGM for the coupled bulk-surface problem. In particular,
we show that the ghost-penalty enhanced discrete form $A_h$ is
coercive with respect to a natural discrete energy-norm and that the
condition number of the resulting system matrix scales as
$\mcO(h^{-2})$, irrespective of the position of $\Omega^h$ relative to
the background mesh $\mcT^h$.

\subsection{Norms and coercivity}
\label{ssec:norms-coercivity}

A natural discrete energy-norm for the forthcoming stability analysis
is given by combining the individual discrete energy norms
for the bulk and surface parts,
\begin{align}
  \label{eq:ahb-norm-def}
  \tn \vb \tnb^2 &= \| \nabla \vb \|^2_{\Omega^h} + \| \vb \|^2_{\Omega^h} 
                 +  \|h^{-\onehalf} [\vb] \|^2_{\mcF^h}
                 + \jbh(\vb,\vb),
                 \\
  \label{eq:ahs-norm-def}
  \tn \vs \tns^2 &= \| \nablash \vs \|^2_{\Gamma^h} + \| \vs \|^2_{\Gamma^h}
                 +  \|h^{-\onehalf} [\vs] \|^2_{\mcE^h}
                 + \jsh(\vs,\vs),
  \\
  \intertext{with the semi-norm induced by the coupling bilinear $a^h_{\Omega\Gamma}$
  to define}
  \tn v \tn_h^2 &= \cb \tn \vb \tnb^2 + \cs \tn \vs \tns^2 + \| \cb \vb - \cs \vs \|_{\Gamma^h}^2.
  \label{eq:ah-norm-def}
\end{align}
%
With these norm definitions, the coercivity of the total bilinear form
$A^h$ can be easily shown once coercivity properties for the bulk and surface
bilinear form are established individually. In other words, 
we wish to show that
\begin{alignat}{3}
  \tn \vb \tnb^2 &\lesssim
  \Ab(\vb, \vb) 
  & &\quad \foralls \vb \in \Vb^h,
  \label{eq:ahb-coerc}
  \\
  \tn \vs \tns^2
  &\lesssim
  \As(\vs, \vs) 
  & & \quad \foralls \vs \in \Vs^h,
  \label{eq:ahs-coerc}
\end{alignat}
which together with the simple observation that
  \begin{align}
    A^h(v,v) &= \Ab^h(\vb,\vb) + \As^h(\vs,\vs) + \abs^h(v,v)
               \\
               &\gtrsim
               \tn \vb \tnb^2 
               + \tn \vs \tns^2
               +\| \cb \vb - \cs \vs \|_{\Gamma^h}^2,
\end{align}
leads us to the following proposition.
\begin{proposition}
  The discrete bilinear form $A^h$ is coercive with respect
  to the discrete energy norm~(\ref{eq:ah-norm-def}):
  \begin{align}
    \tn v \tn_h^2 &\lesssim A^h(v,v), \quad \foralls v \in V^h.
  \end{align}
\end{proposition}
The following two subsection are thus devoted to prove that the estimates~(\ref{eq:ahb-coerc}) and~(\ref{eq:ahs-coerc})
hold.

\subsection{Coercivity of the discrete bulk form $\Ab^h$}
\label{ssec:coerc-Abh}
A standard ingredient in the numerical analysis of discontinuous Galerkin methods is
the inverse inequality
\begin{align}
  \| n_F \cdot \nabla v \|_F \leqslant C_I h_T^{-\onehalf}\| \nabla v \|_T,
  \label{eq:inverse-trace-ineq}
\end{align}
which holds for discrete functions $v \in P_1(T)$. 
Here, the face $F$ is part of the element boundary $\partial T$ and
the inverse constant $C_I = C_I(\tfrac{|F|}{|T|})$ depends on the ratio
of the face area $|F|$ and element volume $|T|$, and thus ultimately on the
shape regularity of $\mcT^h$.
Unfortunately, a
corresponding inverse inequality of the form
\begin{align}
  \| n_F \cdot \nabla v \|_{F\cap\Omega^h} \leqslant C_I h_T^{-\onehalf}\| \nabla v \|_{T \cap \Omega^h}
\end{align}
does not hold as the ratio $\tfrac{|F|}{|T|}$ can become arbitrarily large,
depending on the cut configuration.
As a partial replacement, one might be tempted to use the simple estimate 
\begin{align}
  \| n_F \cdot \nabla v \|_{F\cap \Omega^h}
  \leqslant 
  \| n_F \cdot \nabla v \|_{F}
  \leqslant C_I h_T^{-\onehalf}\| \nabla v \|_T
  \label{eq:inverse-est-simple}
\end{align}
instead.
To fully exploit this idea, it is necessary to extend the control
of the $\|\nabla v\|^2_{\Omega^h}$ part in natural energy norm associated with~$\ab^h$
from the physical domain $\Omega^h$ to the entire active mesh $\mcTb^h$.
This is precisely the role of the ghost-penalty term~$\jbh$:
\begin{lemma}
  \label{lem:ghost-penalty-bulk}
  For $v \in \Vb^h$ it holds that
  \begin{gather}
    \| \nabla v \|_{\mcTb^h}^2
    \lesssim
    \| \nabla v \|_{\Omega^h}^2  + \jbh(v, v)
    \lesssim
    \| \nabla v \|_{\mcTb^h}^2,
    \intertext{and consequently, using~\eqref{eq:inverse-est-simple}}
    \| h^{\onehalf} n_F \cdot \nabla v \|_{\mcFb^h\cap\Omega^h}^2
    \lesssim
    \| \nabla v \|_{\Omega^h}^2  + \jbh(v, v),
    \label{eq:normal-flux-estimate}
  \end{gather}
  with the hidden constant depending only in the shape-regularity of $\mcT^h$.
\end{lemma}
\begin{proof}
 For a detailed proof, we refer to~\cite{BurmanHansbo2012, Massing2012,MassingLarsonLoggEtAl2013a}.
\end{proof}
Thanks to the ghost penalty Lemma~\ref{lem:ghost-penalty-bulk},
we can establish the coercivity of $\Ab^h$ by simply following the standard
arguments in the classical proof for symmetric interior penalty methods.
\begin{proposition}
  \label{prop:coerc-bulk}
  The discrete bulk form $\Ab^h$ is coercive with respect
  to the discrete energy norm $\tn \cdot \tnb$; that is,
  \begin{gather}
    \tn v \tnb^2 \lesssim \Ab^h(v, v), \quad \foralls v \in \Vb^h,
  \label{eq:coercivity-Ahb}
  \end{gather}
\end{proposition}
\begin{proof}
  We follow closely the standard arguments. Setting $\wb = \vb$ in~(\ref{eq:ah-bulk-def})
  and combining the ghost-penalty Lemma~\ref{lem:ghost-penalty-bulk}
  and a $\epsilon$-Cauchy-Schwarz inequality of the form $2 ab \leqslant \epsilon a^2 + \epsilon^{-1} b^2$
  with an inverse estimate
  yields
  \begin{align}
    \Ab^h(v, v)
    &= \| \nabla v \|_{\Omega^h}^2
      - 2 (\mean{n_F \cdot \nabla v}, \jump{v})_{\mcFb^h \cap \Omega^h}
      + \gamma_{\Omega} \|h^{-\onehalf}\jump{v}\|_{\mcFb^h}^2
      \\
    &\quad
      + \jbh(v,v) + \| v \|_{\Omega^h}^2
    \\
    &\gtrsim
     \| \nabla v \|_{\mcTb^h}^2
      - \epsilon \|h^{\onehalf}\mean{n_F \cdot \nabla \vs}\|_{\mcFb^h}^2
      -\epsilon^{-1}\|h^{-\onehalf} \jump{v}\|_{\mcFb^h} \|^2
    \\
    &\quad
      + \gamma_{\Omega} \|h^{-\onehalf} \jump{v}\|_{\mcFb^h}^2
      + \onehalf \jbh(v,v) + \| v \|_{\Omega^h}^2
    \\
    &\gtrsim
     (1-\epsilon C_I)\| \nabla v \|_{\mcTb^h}^2
    \\
    &\quad
      + (\gamma_{\Omega} - \epsilon^{-1})\| h^{-\onehalf}\jump{v}\|_{\mcFb^h}^2
      + \onehalf \jbh(v,v) + \| v \|_{\Omega^h}^2
      \gtrsim \tn v \tnb^2
  \end{align}
  if we chose $0 < \epsilon \lesssim 1/(2C_I)$ small enough and $\gammab > \epsilon^{-1}$.
\end{proof}

\subsection{Coercivity of the discrete surface form $\As^h$}
\label{ssec:coerc-discr-surf}
Next, we turn to the stability properties of the discrete surface form $\As^h$.
First observe that the unstabilized DG energy ``norm''
\begin{align}
\tn v \tn_{\Gamma}^2 \coloneqq \| \nablash v \|_{\Gammah}^2 + \|v \|_{\Gammah}^2
  + \|h^{-\onehalf} \jump{v} \|_{\mcE^h}^2
  \label{eq:unstabilized-dg-surface-norm}
\end{align}
does not define an actual norm on $\Vs^h$. For instance, the piecewise linear and continuous
approximation $\rho^h$ of the distance function $\rho$ vanishes on $\Gammah$.
It was shown in \cite{BurmanHansboLarsonEtAl2016a} that a proper norm can obtained if
the ghost penalty term $\jsh$ was added, resulting in our norm definition~(\ref{eq:ahs-norm-def}).
More, precisely, the following discrete Poincar\'e inequality was established.
\begin{lemma}
  \label{lem:discrete-poincare}
  Let $h \in (0,h_0]$ with $h_0$ small
  enough. Then the following estimate holds:
  \begin{equation}
    h^{-1}\| v - \lambda_{\Gamma^h}(v) \|^2_{\mcT^h}
    \lesssim
    \| \nablash v \|^2_{\Gamma^h}
    + \jsh(v,v)
    \quad \forall v \in V^h,
  \end{equation}
  where $\lambda_{\Gamma^h}(v) = \tfrac{1}{|\Gammah|}\int_{\Gammah} v \dsh$
  is the mean value of $v$ on $\Gammah$.
\end{lemma}
To prove that $\As^h$ is in fact coercive with respect to a
properly defined discrete energy norm,
we need to borrow one more result from~\cite{BurmanHansboLarsonEtAl2016a}
which allows us to control the co-normal flux~$n_E \cdot \nablash v$
for $v\in \Vs^h$.

\begin{lemma} The following estimate holds
\begin{equation}
  h \| \nablash v \|^2_{\partial \mcK_h}
\lesssim
\| \nablash v \|^2_{\Gamma^h}
+\jsh(v,v),
\label{eq:inverse-estimate-for-conormal-flux}
\end{equation}
for $0< h \leq h_0$ with $h_0$ small enough.
\label{lem:inverse-estimate-for-conormal-flux}
\end{lemma}
Now simply replacing the crucial normal-flux estimate~\eqref{eq:normal-flux-estimate}
with co-normal flux estimate from the previous
Lemma~(\ref{lem:inverse-estimate-for-conormal-flux}),
the proof of Lemma~\ref{prop:coerc-bulk} literally transfers to the surface case,
and thus we have established
the following result.
\begin{proposition}
  The discrete surface form $\As^h$ is coercive with respect
  to the discrete energy norm $\tn \cdot \tns$:
  \begin{align}
    \tn v \tns^2 &\lesssim \As^h(v, v), \quad \foralls v \in \Vs^h.
  \end{align}
\end{proposition}
\subsection{Condition number estimates}
\label{ssec:condition-number-estimate}

Following closely the presentation in~\cite{BurmanHansboLarsonEtAl2014},
we now show that the condition number of the system matrix
associated with a properly rescaled version of the bilinear form~(\ref{eq:a-def})
can be bounded by $O(h^{-2})$ independently of the position of the
bulk domain $\Omega$ relative to the background mesh~$\mcT_h$.
Let $\{\phi_{\Omega,i}\}_{i=1}^{N_{\Omega}}$
and $\{\phi_{\Gamma,i}\}_{i=1}^{N_{\Gamma}}$
be the standard piecewise linear basis
functions associated with $\mcTb^h$ and $\mcTs^h$, respectively.
Thus
\begin{align}
v^h = (\vs^h, \vb^h)
= \Bigl(\sum_{i=1}^{N_{\Omega}} V_{\Omega,i} \phi_{\Omega,i}, \sum_{i=1}^{N_{\Gamma}} V_{\Gamma,i} \phi_{\Gamma,i}\Bigr)
\end{align}
for $v^h \in V^h$ and expansion
coefficients
$V = \bigl(\{V_{\Omega,i}\}_{i=1}^{N_{\Omega}}, \{V_{\Gamma,i}\}_{i=1}^{N_{\Gamma}}\bigl)
\in \RR^{N_{\Omega}}\times\RR^{N_{\Gamma}} = \RR^N$ with $N =N_{\Omega}+ N_{\Gamma}$.
It is well-known that
for any quasi-uniform mesh $\mcT^h$ consisting
of $d$-dimensional simplices,
the continuous $\| \cdot \|_{L^2(\mcT^h)}$- norm
of a finite element function $v \in \mcV^h = \spann(\{\phi_i\}_{i=1}^{M})$
is related to
the discrete $\| \cdot\|_{l^2(\RR^M)}$
of its corresponding coefficient vector $V$
via
\begin{align}
  h^{d/2} \| V \|_{\RR^M}
  \lesssim \| v_h \|_{L^2(\mcT^h)}
  \lesssim
  h^{d/2} \| V \|_{\RR^M}.
  \label{eq:mass-matrix-scaling}
\end{align}
Note that due to the different Hausdorff dimensions of the surface and bulk domain,
the discrete norms and forms for each domain scale differently with respect to the mesh size $h$.
For instance, we have clearly the Poincar\'e-type estimate\footnote{This is trivial
  since the mass term is already included in our form.}
\begin{align}
  \|\vb\|_{\mcTb^h} \lesssim \tn \vb \tnb \quad \foralls \vb  \in \Vb^h,
  \label{eq:poincare-est-bulk}
\end{align}
while for the surface problem, Lemma~\ref{lem:discrete-poincare} shows that
we have 
\begin{align}
\|\vs\|_{\mcTs^h} \lesssim \tn h^{\onehalf}\vs \tns \quad \foralls \vs  \in \Vs^h.
  \label{eq:poincare-est-surface}
\end{align}
Thus in order to pass back and forth between discrete $l^2$ and continuous,
similarly scaled $L^2$ norms on the surface and in the bulk domain,
it is natural to rescale the discrete surface functions. More precisely,
the system matrix $\mcA$ we will consider is given by the relation
\begin{align}
  ( \mcA V, W )_{\RR^N}  = \widetilde{A}_h(v_h, w_h) \coloneqq
  {A}_h(v_h, h^{\onehalf} w_h)
  \quad \foralls v_h,w_h \in
  V_h.
  \label{eq:system-matrix}
\end{align}
The system matrix $\mcA$ is
a bijective linear mapping 
$\mcA:\RR^N \to \RR^N$.
The operator norm and condition number of the matrix $\mcA$ are then defined by
\begin{align}
  \| \mcA \|_{\RR^N}
  = \sup_{V \in \widehat{\RR}^N\setminus 0}
  \dfrac{\| \mcA V \|_{\RR^N}}{\|V\|_N}
\quad \text{and}
\quad
  \kappa(\mcA) = \| \mcA \|_{\RR^N} \| \mcA^{-1} \|_{\RR^N}
  \label{eq:operator-norm-and-condition-number-def}
\end{align}
respectively.
Following the approach in~\cite{ErnGuermond2006}, 
a bound for the condition number can be derived 
by combining~\eqref{eq:mass-matrix-scaling}
with suitable Poincar\'e-type estimates and inverse estimates
relating the $L^2$ norm to the discrete energy norms.
An immediate consequence of the discrete Poincar\'e estimates
for the discrete bulk and surface energy norms given by~\eqref{eq:poincare-est-bulk}
and~\eqref{eq:poincare-est-surface}, respectively, is the following
Poincar\'e estimate for the total discrete energy norm:
\begin{lemma}
  For $(\vb, \vs) \in V^h = \Vb^h \times \Vs^h$ it holds
  \begin{align}
    \| (\vb, \vs)\|_{\mcTb^h \times \mcTs^h} \lesssim
    \tn  (\vb, h^{\onehalf}\vs) \tn_h.
    \label{eq:discrete-poincare-Ah}
  \end{align}
\end{lemma}
Before we turn to  formulate and prove a suitable inverse inequality
for the total discrete energy norm, we briefly recall that
we have the following inverse inequalities:
\begin{alignat}{3}
  \| \nabla v \|_{T} &\lesssim h^{-1} \| v \|_{T},  && \quad \foralls v \in \Vb^h,
    \label{eq:inverse-est-element}
    \\
  \| v \|_{F} &\lesssim h^{-\onehalf} \| v \|_{T},  && \quad \foralls v \in \Vb^h,
  \label{eq:inverse-est-face}
  \\
  \| v \|_{\Gamma^h \cap T} &\lesssim h^{-\onehalf} \| v \|_{T}, && \quad \foralls v \in \Vs^h.
    \label{eq:inverse-est-surface}
\end{alignat}
While the first two are standard, the third one is less known and can be found
in, e.g., \cite{BurmanHansboLarson2015,BurmanHansboLarsonEtAl2016a,BurmanHansboLarsonEtAl2016c,BurmanHansboLarsonEtAl2016}. Now it is easy to show the following inverse inequality.
\begin{lemma}
  \label{lem:inverse-estimate-Ah}
  For $(\vb, \vs) \in V^h = \Vb^h \times \Vs^h$ it holds
  \begin{align}
    \tn  (\vb, h^{\onehalf}\vs) \tn_h \lesssim
    h^{-1}\| (\vb, \vs)\|_{\mcTb^h \times \mcTs^h}.
    \label{eq:inverse-estimate-Ah}
  \end{align}
\end{lemma}
  \begin{proof}
    Recalling the definition of $\tn \cdot \tn_h$,
    \begin{align}
      \tn (\vb, h^{\onehalf}\vs)\tn_h^2 &= \tn \vb \tnb^2 +  \tn h^{\onehalf}\vs \tns^2
      + \|\cb \vb - h^{\onehalf} \cs \vs\|_{\Gamma^h}^2
      \\
      &= I + II + III,
    \end{align}
    it is enough to consider the last two terms, as term $I$ 
    can be treated similar to $II$.
    We start with the contributions of $II$ which are not related
    to $\jsh$ and after successively applying variants of the inverse estimates
    type~(\ref{eq:inverse-est-element}), (\ref{eq:inverse-est-surface}), we get
    \begin{align}
      \| h^{\onehalf} \nablash \vs \|_{\Gamma^h}
      &\lesssim
      \| \nabla \vs \|_{\mcTs^h} \lesssim
      h^{-1} \| \vs \|_{\mcTs^h},
      \\
      \| h^{\onehalf} \vs \|_{\Gamma^h}
      &\lesssim
      \| \vs \|_{\mcTs^h},
      \\
      \|h^{-\onehalf} \jump{h^{\onehalf} \vs} \|_{\mcE^h \cap T}
      &\lesssim h^{-1} \| \vs \|_{\mcTs^h}.
    \end{align}
    Turning to the contribution from $\jsh$, we see that
    \begin{align}
      \jsh(h^{\onehalf} \vs, h^{\onehalf} \vs)^{\onehalf}
      &\lesssim
        h^{-1}\|\jump{h^{\onehalf}  \vs} \|_{\mcFs^h}
        + \|\jump{h^{\onehalf} n_F \cdot \nabla \vs} \|_{\mcFs^h}
        \lesssim h^{-1} \| \vs \|_{\mcTs^h}.
    \end{align}
    Finally, we conclude the proof by estimating the remaining term $III$ as follows,
    \begin{align}
      III \lesssim
      \|\cb \vb \|_{\Gamma^h} + 
      \|h^{\onehalf} \cs \vs\|_{\Gamma^h}
      \lesssim  
      h^{-\onehalf} \|\cb \vb \|_{\mcTb^h} 
      + \|\cs \vs \|_{\mcTs^h}.
    \end{align}
  \end{proof}
\begin{theorem} 
  \label{thm:condition-number-estimate}
  The condition number of the stiffness matrix satisfies
  the estimate
\begin{equation}
\kappa( \mcA )\lesssim h^{-2},
\end{equation}
  where the hidden constant depends only on the quasi-uniformity
  of the  background mesh $\mcT^h$ and the chosen stability parameters.
\end{theorem}
\begin{proof} We need to bound $\| \mcA \|_{\RR^N}$ and $\| \mcA^{-1} \|_{\RR^N}$. 
  To derive a bound for $\| \mcA \|_{\RR^N}$, 
  we first use the inverse estimate~\eqref{eq:inverse-estimate-Ah}
  and equivalence~\eqref{eq:mass-matrix-scaling} to find that
$\foralls w \in V^h$,
\begin{equation}
\tn (\wb, h^{\onehalf} \ws) \tn_h 
\lesssim h^{-1} \| (\wb, \ws) \|_{\mcTb^h \times \mcTs^h}
\lesssim h^{(d-2)/2}\|W\|_{\RR^N}.
\end{equation}
Then
\begin{align}
  \| \mcA V\|_{\RR^N} &= \sup_{W \in \RR^N } 
  \frac{( \mcA V, W)_{\RR^N}}{\| W \|_{\RR^d}}
  = \sup_{w \in V_h }  \frac{\widetilde{A}^h(v,w)}{\tn(\wb, h^{\onehalf} \ws)\tn_h}
  \frac{\tn (\wb, h^{\onehalf} \ws) \tn_h}{\| W \|_{\RR^N}}
\\
  &\lesssim h^{(d-2)/2} \tn (\vb, h^{\onehalf}\vs) \tn_h 
\lesssim h^{d-2}\| V \|_N,
\end{align}
and thus by the definition of the operator norm,
$
\| \mcA \|_{\RR^N} \lesssim h^{d-2}
$.
Next we turn to the estimate of $\| \mcA^{-1}\|_{\RR^N}$.
Starting from \eqref{eq:mass-matrix-scaling} and combining the Poincar\'e
inequality~\eqref{eq:discrete-poincare-Ah} with the stability
estimates~\eqref{eq:coercivity-Ahb} and
a Cauchy Schwarz inequality, we arrive at the following chain of
estimates:
\begin{align}
  \| V \|^2_{\RR^N} 
  &\lesssim h^{-d} \| (\vb, \vs) \|^2_{\Omega^h\times\Gamma^h} 
  \lesssim h^{-d} \tn (\vb, h^{\onehalf}\vs) \tn_h^2
  \\
  &\lesssim h^{-d} \widetilde{A}_h(v,v) 
  = h^{-d} (V, \mcA V)_{\RR^N}
  \lesssim h^{-d} \| V \|_{\RR^N} \| \mcA V \|_{\RR^N},
\end{align}
and hence $\| V \|_{\RR^N} \lesssim h^{-d}\| \mcA V\|_{\RR^N}$. 
Now setting $ V = \mcA^{-1} W$ we conclude that
 $
 \| \mcA^{-1}\|_{\RR^N} \lesssim h^{-d}
 $
and combining the estimates for $\| \mcA\|_{\RR^N}$ and $\| \mcA^{-1}\|_{\RR^N}$ the theorem follows.
\end{proof}
\begin{figure}[htb]
  \begin{center}
    \begin{minipage}[t]{0.45\textwidth}
      \vspace{0pt}
    \includegraphics[width=1.0\textwidth]{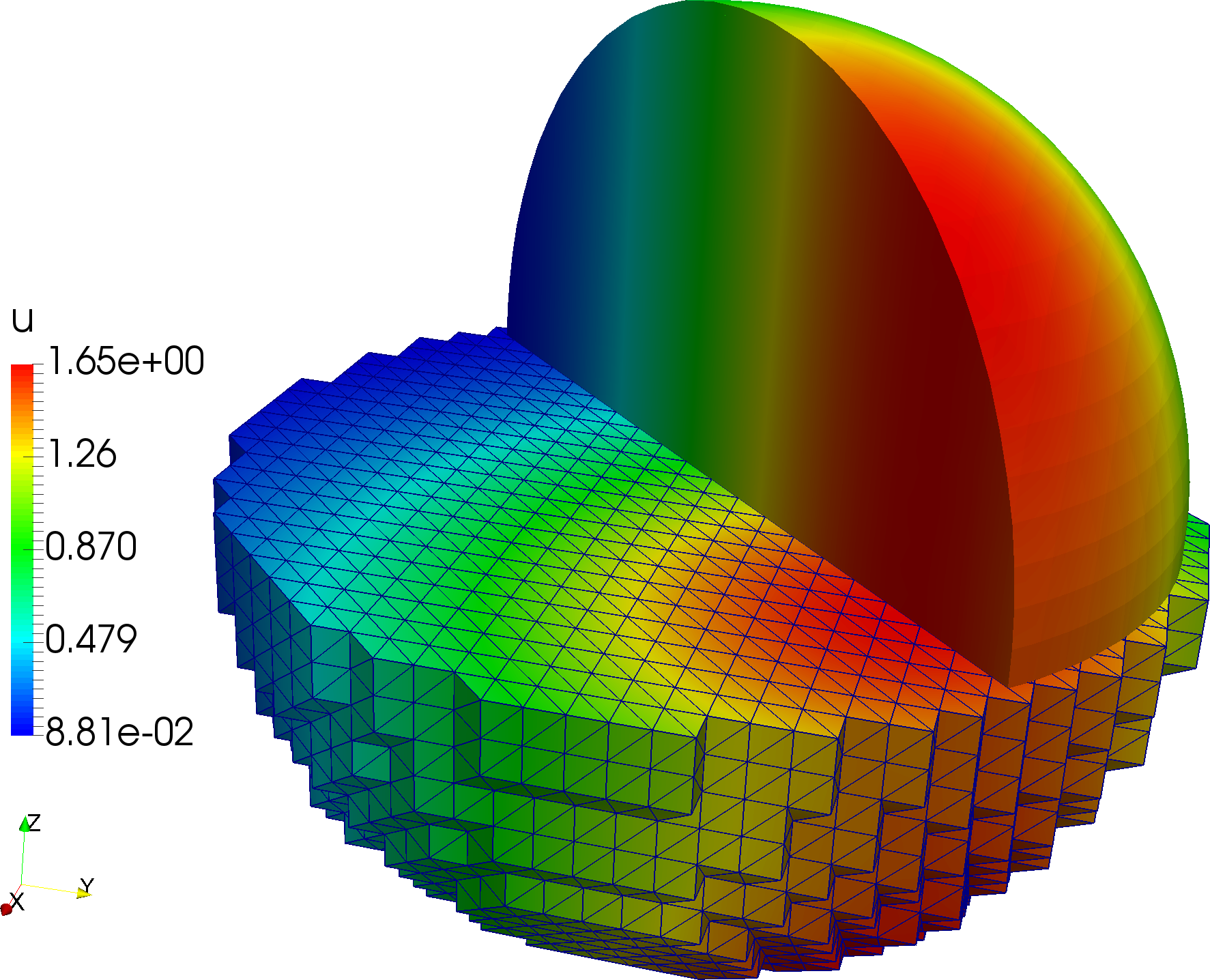}
  \end{minipage}
  \hspace{0.02\textwidth}
    \begin{minipage}[t]{0.45\textwidth}
      \vspace{0pt}
    \includegraphics[width=1.0\textwidth]{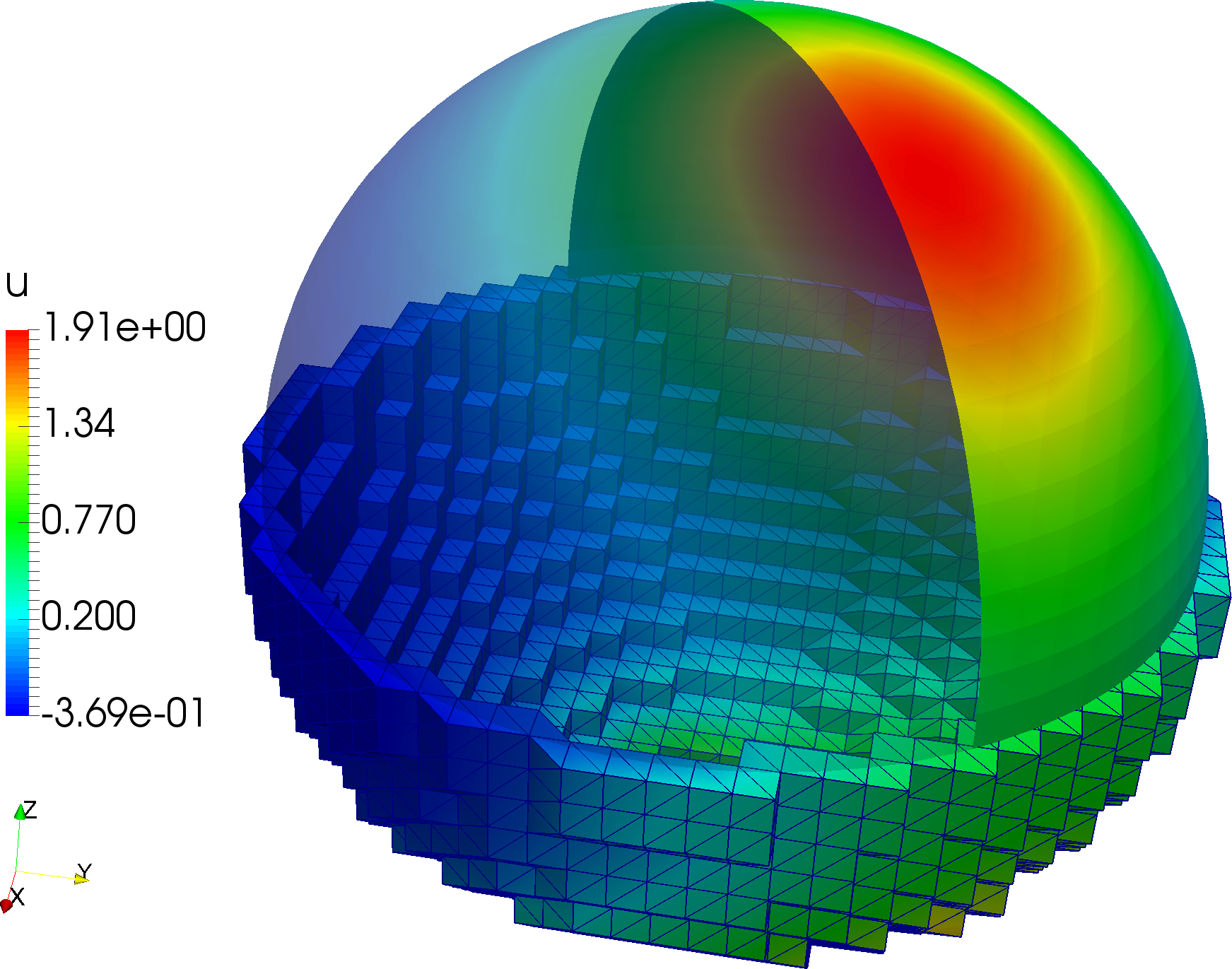}
      \vspace{0pt}
    \end{minipage}
    \caption{Computed solutions for coupled bulk-surface PDE example. 
      The left plot shows the approximate bulk
      solution $\us^h$ as computed on the active mesh $\mcTb^h$,
      together with it restriction to the bulk domain $\Omega^h$.
      The right plot displays the corresponding surface solution
      $\us^h$.}
    \label{fig:solution-example-1}
  \end{center}
\end{figure}

\section{Numerical results}
\subsection{Convergence rate study}
\label{ssec:conv-rate-study}
Following the numerical example presented in~\cite{ElliottRanner2013},
we now examine the convergence properties of the presented cutDG method
for the bulk-surface problem~(\ref{eq:bulk-surface-strong}).
An analytical reference solution is defined by
\begin{align}
  \ub(x, y, z) &= \cs e^{-x(x - 1)y(y-1)}, \\
  \us(x, y, z) &= (\cs + x(1-2x) + y(1-2x))e^{-x(x - 1)y(y - 1)},
\end{align}
with $\cb = \cs = 1$, the corresponding the right-side $f = (\fb,
\fs)$ is computed such that $u = (\us, \ub)$
satisfies~(\ref{eq:bulk-strong})--~(\ref{eq:surface-strong}).
Starting from a structured background mesh $\widetilde{\mcT}_0$ for
$\Omega = [-1.1,1.1]^3$, a sequence of meshes $\{\mcT_k\}_{k=0}^5$ is
generated by successively refining $\widetilde{\mcT}_0$ and extracting
the relevant active background meshes for the bulk and surface problem
as defined by~(\ref{eq:active-mesh-bulk})--(\ref{eq:active-mesh-surface}).
Based on the manufactured exact solution, the experimental order of
convergence (EOC) is calculated by
\begin{align*}
    \text{EOC}(k) = \dfrac{\log(E_{k-1}/E_{k})}{\log(2)}
\end{align*}
with $E_k$ denoting the (norm-dependent) error of the numerical
solution $u_k$ computed at refinement level $k$.
In the present convergence study, both $\| \cdot \|_{H^1(U)}$
and $\|\cdot\|_{L^2(U)}$ for $U \in \{\Omega^h, \Gamma^h\}$
are used to compute $E_k$.
For the completely stabilized cutDG method with
$\gamma_{\Omega} = \gamma_{\Gamma} = 50$,
$\mu_{\Omega} =\mu_{\Gamma} = 50$ and
$\tau_{\Omega} = \tau_{\Gamma} = 0.01$, 
the observed EOC reported in Table~\ref{tab:convergence-rates-sbp-example-1}
(top)
reveals a first-order and second-order
convergence in the $H^1$ and $L^2$ norm, respectively.
Note that for the bulk problem, the standard DG jump penalization term in~(\ref{eq:ah-bulk-def})
scaled with $\gamma_{\Omega}$ is similar to the
lowest order term in the ghost-penalty~(\ref{eq:jbh-def})  scaled with $\mu_{\Omega}$.
Deactivating all solely ghost-penalty related stabilization
by setting $\tau_{\Omega} =\tau_{\Gamma} = \mu_{\Gamma} = 0$
renders the method completely unreliable and thus demonstrates
the necessity to stabilize the presented DG method for the bulk-surface problem
in the unfitted mesh case.
\begin{table}[htb]
  \centering
  \begin{center}
\begin {tabular}{cr<{\pgfplotstableresetcolortbloverhangright }@{}l<{\pgfplotstableresetcolortbloverhangleft }cr<{\pgfplotstableresetcolortbloverhangright }@{}l<{\pgfplotstableresetcolortbloverhangleft }cr<{\pgfplotstableresetcolortbloverhangright }@{}l<{\pgfplotstableresetcolortbloverhangleft }cr<{\pgfplotstableresetcolortbloverhangright }@{}l<{\pgfplotstableresetcolortbloverhangleft }c}%
\toprule $k$&\multicolumn {2}{c}{$\| e^h \|_{H^1(\Omega ^h)}$}&EOC&\multicolumn {2}{c}{$\| e^k \|_{L^2(\Omega ^h)}$}&EOC&\multicolumn {2}{c}{$\| e^k \|_{H^1(\Gamma ^h)}$}&EOC&\multicolumn {2}{c}{$\| e^k \|_{L^2(\Gamma ^h)}$}&EOC\\\midrule %
\pgfutilensuremath {0}&$5.28$&$\cdot 10^{-1}$&--&$8.60$&$\cdot 10^{-2}$&--&$2.17$&$\cdot 10^{0}$&--&$2.73$&$\cdot 10^{-1}$&--\\%
\pgfutilensuremath {1}&$3.44$&$\cdot 10^{-1}$&\pgfutilensuremath {+0.62}&$3.04$&$\cdot 10^{-2}$&\pgfutilensuremath {+1.50}&$1.12$&$\cdot 10^{0}$&\pgfutilensuremath {+0.96}&$7.38$&$\cdot 10^{-2}$&\pgfutilensuremath {+1.89}\\%
\pgfutilensuremath {2}&$1.84$&$\cdot 10^{-1}$&\pgfutilensuremath {+0.90}&$7.34$&$\cdot 10^{-3}$&\pgfutilensuremath {+2.05}&$5.80$&$\cdot 10^{-1}$&\pgfutilensuremath {+0.94}&$1.80$&$\cdot 10^{-2}$&\pgfutilensuremath {+2.04}\\%
\pgfutilensuremath {3}&$9.35$&$\cdot 10^{-2}$&\pgfutilensuremath {+0.98}&$1.83$&$\cdot 10^{-3}$&\pgfutilensuremath {+2.00}&$2.76$&$\cdot 10^{-1}$&\pgfutilensuremath {+1.07}&$4.63$&$\cdot 10^{-3}$&\pgfutilensuremath {+1.96}\\%
\pgfutilensuremath {4}&$4.71$&$\cdot 10^{-2}$&\pgfutilensuremath {+0.99}&$4.66$&$\cdot 10^{-4}$&\pgfutilensuremath {+1.98}&$1.39$&$\cdot 10^{-1}$&\pgfutilensuremath {+0.99}&$1.07$&$\cdot 10^{-3}$&\pgfutilensuremath {+2.12}\\\bottomrule %
\end {tabular}%

    \\[2ex]
\begin {tabular}{cr<{\pgfplotstableresetcolortbloverhangright }@{}l<{\pgfplotstableresetcolortbloverhangleft }cr<{\pgfplotstableresetcolortbloverhangright }@{}l<{\pgfplotstableresetcolortbloverhangleft }cr<{\pgfplotstableresetcolortbloverhangright }@{}l<{\pgfplotstableresetcolortbloverhangleft }cr<{\pgfplotstableresetcolortbloverhangright }@{}l<{\pgfplotstableresetcolortbloverhangleft }c}%
\toprule $k$&\multicolumn {2}{c}{$\| e^h \|_{H^1(\Omega ^h)}$}&EOC&\multicolumn {2}{c}{$\| e^k \|_{L^2(\Omega ^h)}$}&EOC&\multicolumn {2}{c}{$\| e^k \|_{H^1(\Gamma ^h)}$}&EOC&\multicolumn {2}{c}{$\| e^k \|_{L^2(\Gamma ^h)}$}&EOC\\\midrule %
\pgfutilensuremath {0}&$7.27$&$\cdot 10^{-1}$&--&$1.32$&$\cdot 10^{-1}$&--&$5.38$&$\cdot 10^{0}$&--&$1.16$&$\cdot 10^{0}$&--\\%
\pgfutilensuremath {1}&$8.88$&$\cdot 10^{-1}$&\pgfutilensuremath {-0.29}&$1.99$&$\cdot 10^{-1}$&\pgfutilensuremath {-0.59}&$8.46$&$\cdot 10^{0}$&\pgfutilensuremath {-0.65}&$1.82$&$\cdot 10^{0}$&\pgfutilensuremath {-0.65}\\%
\pgfutilensuremath {2}&$1.14$&$\cdot 10^{0}$&\pgfutilensuremath {-0.36}&$2.72$&$\cdot 10^{-1}$&\pgfutilensuremath {-0.45}&$1.02$&$\cdot 10^{2}$&\pgfutilensuremath {-3.59}&$2.60$&$\cdot 10^{0}$&\pgfutilensuremath {-0.51}\\%
\pgfutilensuremath {3}&$1.01$&$\cdot 10^{0}$&\pgfutilensuremath {+0.17}&$2.51$&$\cdot 10^{-1}$&\pgfutilensuremath {+0.11}&$1.87$&$\cdot 10^{1}$&\pgfutilensuremath {+2.44}&$2.25$&$\cdot 10^{0}$&\pgfutilensuremath {+0.21}\\\bottomrule %
\end {tabular}%

  \end{center}
  \caption{Experimental order of convergence for the bulk-surface problem with
  DG-stabilization parameters
  $\gamma_{\Omega} = \gamma_{\Gamma} = 50$.
  (Top) Optimal convergence rates
  are obtained for completely activated ghost-penalties using
    $\mu_{\Omega} = \mu_{\Gamma} = 50$ and
    $\tau_{\Omega} = \tau_{\Gamma} = 0.01$.
    (Bottom) After 
    deactivation of the ghost penalties by setting
    $\mu_{\Gamma} = \tau_{\Omega} = \tau_{\Gamma} = 0$,
    the convergence rate deteriorates completely and no
    clear trend is observable.
  }
  \label{tab:convergence-rates-sbp-example-1}
\end{table}

\subsection{Condition number study}
In the second numerical experiment, we study the sensitivity
of the condition number of the system matrix defined
by~\eqref{eq:system-matrix}
with respect to relative positioning of $\Gamma$
within the background mesh $\mcT_h$.
Starting from the set-up described in Section~\ref{ssec:conv-rate-study}
and choosing refinement level $k=1$,
a family of surfaces
$\{\Gamma_{\delta}\}_{0\leqslant\delta\leqslant 1}$
is generated 
by translating the unit-sphere $S^2 = \{ x \in \RR^3 : \| x \| = 1 \}$
along the diagonal $(h,h,h)$; that is,
$\Gamma_{\delta} = S^2 + \delta_0 (h,h,h)$ with $\delta \in [0,1]$.
Figure~\ref{fig:condition-number-sbp-set-up} illustrates the experimental
set-up.
\begin{figure}[htb]
  \begin{center}
    \begin{minipage}[c]{0.45\linewidth}
      \vspace{0pt}
      \includegraphics[width=1.0\textwidth]{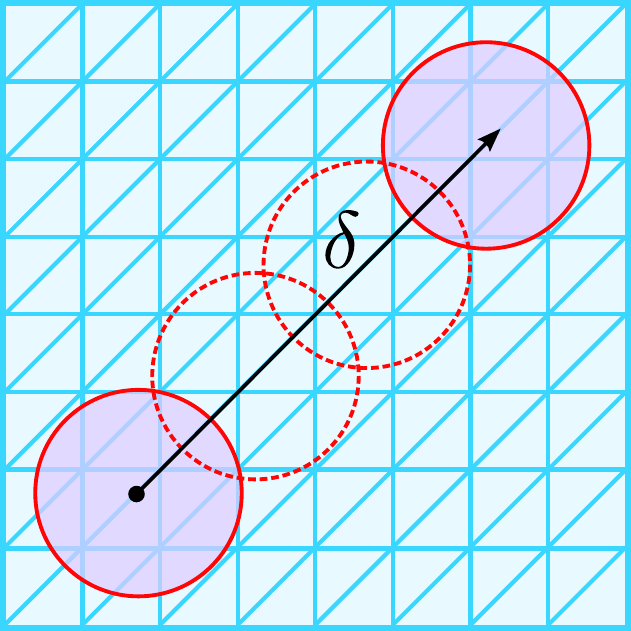}
    \end{minipage}
    \hspace{0.03\linewidth}
    \begin{minipage}[c]{0.5\linewidth}
      \vspace{0pt}
      \includegraphics[width=1.0\textwidth]{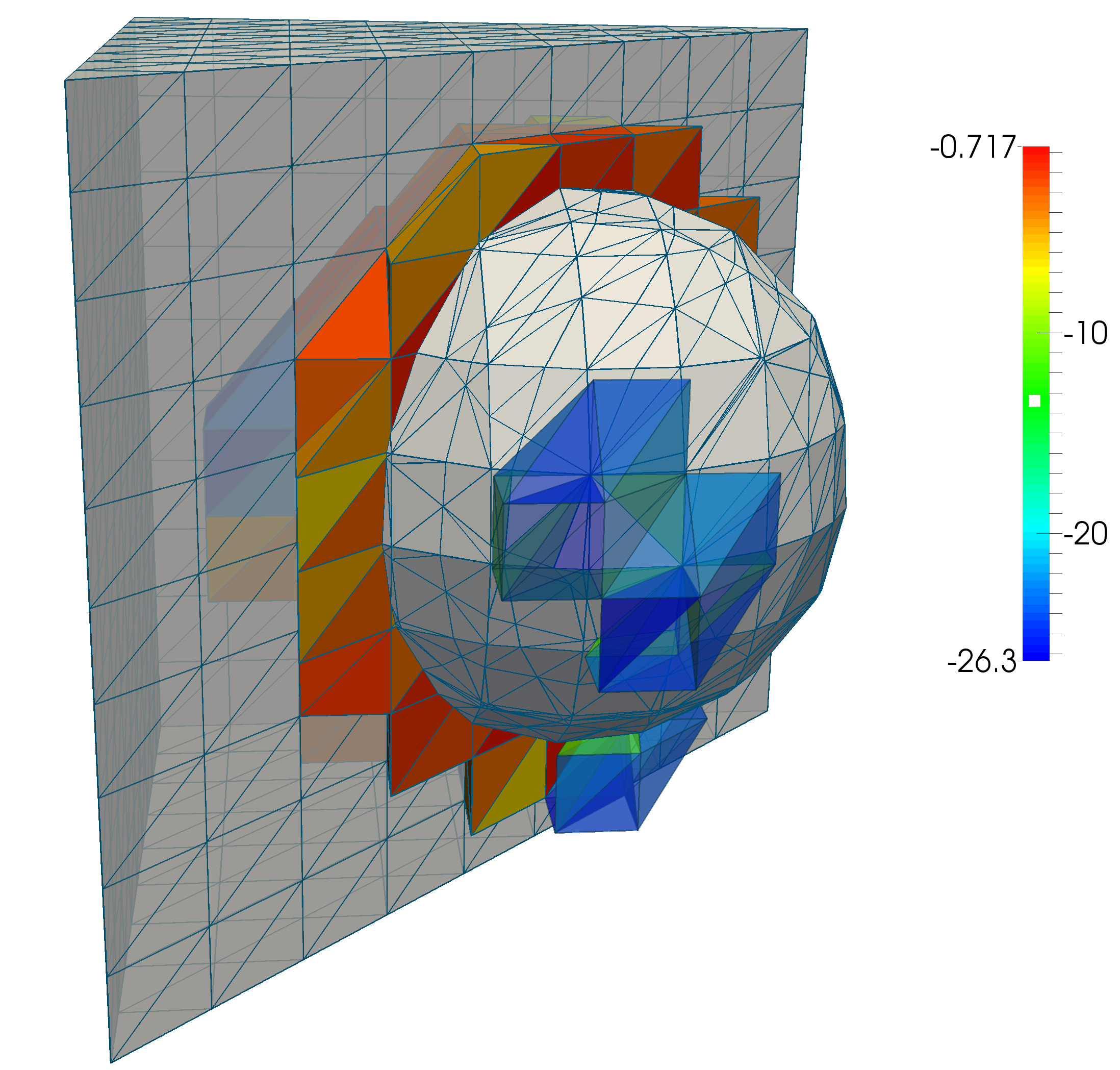}
    \end{minipage}
  \end{center}
  \caption{(Left) Principal experimental set-up to study the sensitivity  of the conditon number
    with respect to the relative $\Gamma$ position.
    (Right): Snapshot of an intersection configuration when moving $\Gamma$
    through the background mesh.
    To visualize ``extreme'' cut configurations,
    the color map plots
    for each intersected mesh element $T$ the value of $\log(\Gamma^h \cap T/\diam(T)^2)$.
    Thus blue-colored elements contain only an extremely small fraction of the surface.
  }
    \label{fig:condition-number-sbp-set-up}
\end{figure}
For $\delta = l/500$, $l=0,\ldots,500$, we compute the condition
number $\kappa_{\delta}(\mcA)$ as the ratio of the absolute value of
the largest (in modulus) and smallest (in modulus), non-zero
eigenvalue.
The resulting condition numbers are displayed 
in Figure~\ref{fig:condition-number-sbp} as a function of $\delta$.
Choosing the stabilization parameters as in the convergence study for
the fully stabilized cutDG method,
we observe that the position of $\Gamma$ relative to the background mesh
$\mcT_k$ has very little effect on the condition number.
After turning off either of the bulk and surface related cutFEM stabilizations,
the condition number is highly sensitive to the relative position of $\Gamma$ and clearly
unbounded as a function of $\delta$.
\begin{figure}[htb]
  \begin{center}
    \includegraphics[width=0.900\textwidth]{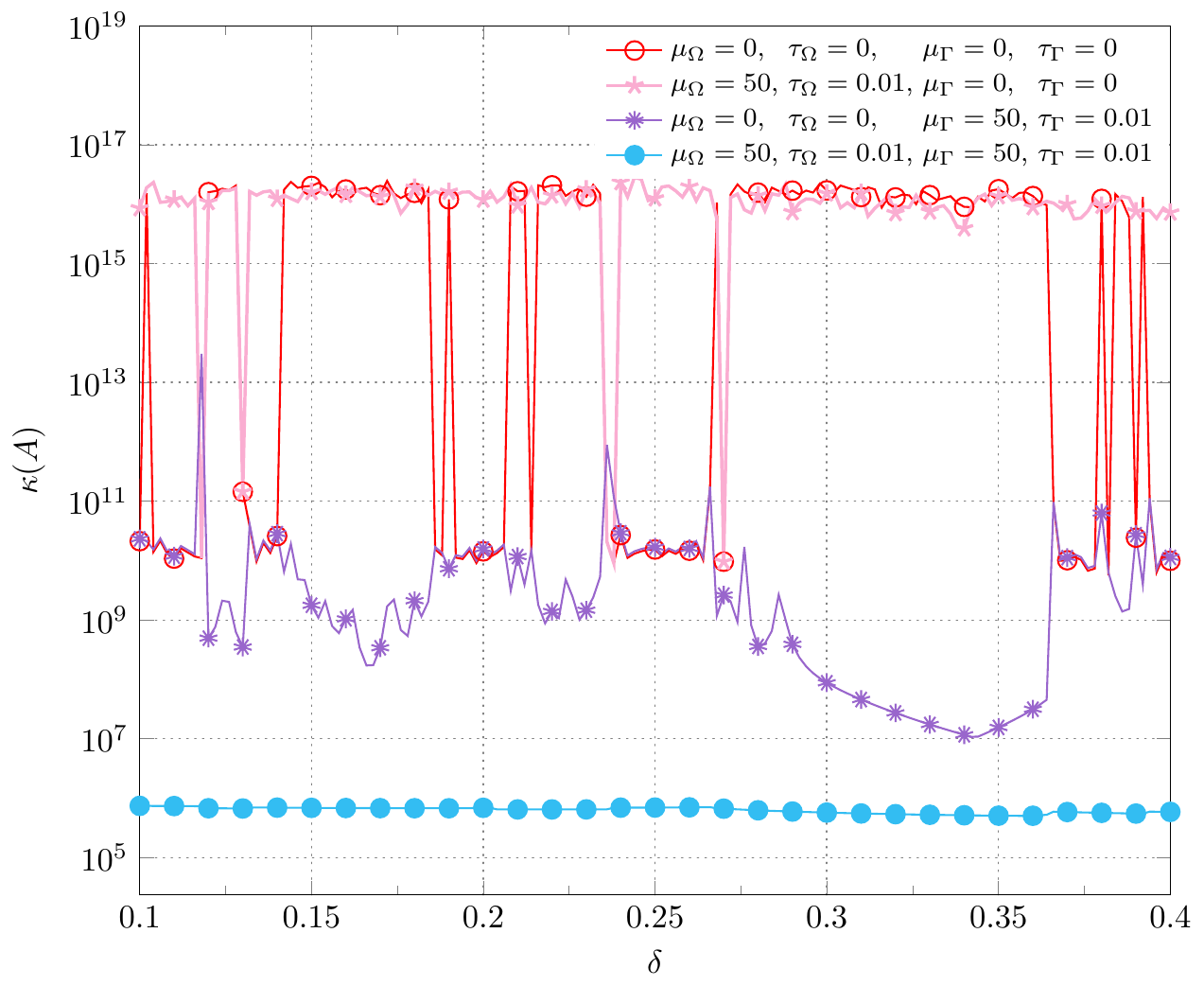}
    \caption{Condition numbers plotted as a function of the position parameter $\delta$.
      When turning off either the surface or bulk related ghost-penalties (or both),
      the condition number
      is highly sensitive to the relative surface positioning in the background mesh.
      With all ghost penalties activated, the condition number is
      completely robust.
  }
    \label{fig:condition-number-sbp}
  \end{center}
\end{figure}

\section{Acknowledgments}
This work was supported in part by the Kempe foundation (JCK-1612).
The author expresses his gratitude to Ceren G\"urkan for her help with the set-up of the convergence
experiment, to Erik Burman for his great editorial assistance during the preparation of this
contribution, and finally, to the two anonymous referees for
their valuable comments and suggestions.

\bibliographystyle{plainnat}
\bibliography{bibliography}

\end{document}